\documentclass{amsart}

\usepackage[all]{xy}

\usepackage{latexsym}

\usepackage{amssymb}
\usepackage{amsfonts}
\usepackage{amscd}
\usepackage{amsmath,amsthm}

\newtheorem{lemma}{Lemma}[section]
\newtheorem{proposition}[lemma]{Proposition}
\newtheorem{theorem}[lemma]{Theorem}

\newtheorem{prop}[lemma]{Proposition}
\newtheorem{thm}[lemma]{Theorem}
\newtheorem{cor}[lemma]{Corollary}

{

}

\theoremstyle{definition}

\newtheorem{example}[lemma]{Example}
\newtheorem{definition}[lemma]{\sl Definition}

\theoremstyle{remark}

\newtheorem{remark}[lemma]{Remark}

\newcommand{\Hom}{\operatorname{Hom}}

\newcommand{\el}{\mathcal{L}}

\newcommand{\arr}{\rightarrow}
\newcommand{\cL}{\mathcal{L}}
\newcommand{\cO}{\mathcal{O}}

\numberwithin{equation}{section}

\begin{document}

\title{Morphisms to noncommutative projective lines}

\author{D. Chan}
\address{University of New South Wales}
\email{danielc@unsw.edu.au}

\author{A. Nyman}
\address{Western Washington University}
\email{adam.nyman@wwu.edu}
\keywords{}
\thanks{2010 {\it Mathematics Subject Classification. } Primary 14A22; Secondary 16S38}

\begin{abstract}
Let $k$ be a field, let ${\sf C}$ be a $k$-linear abelian category, let $\underline{\mathcal{L}}:=\{\mathcal{L}_{i}\}_{i \in \mathbb{Z}}$ be a sequence of objects in ${\sf C}$, and let $B_{\underline{\mathcal{L}}}$ be the associated orbit algebra.  We describe sufficient conditions on $\underline{\mathcal{L}}$ such that there is a canonical functor from the noncommutative space ${\sf Proj }B_{\underline{\mathcal{L}}}$ to a noncommutative projective line in the sense of \cite{abstractp1}, generalizing the usual construction of a map from a scheme $X$ to $\mathbb{P}^{1}$ defined by an invertible sheaf $\mathcal{L}$ generated by two global sections.  We then apply our results to construct, for every natural number $d>2$, a degree two cover of Piontkovski's $d$th noncommutative projective line \cite{piont} by a noncommutative elliptic curve in the sense of Polishchuk \cite{polish1}.
\end{abstract}

\maketitle

\pagenumbering{arabic}

\section{Introduction}
Although an algebraic description of morphisms into projective space is central to algebraic geometry, a similar description in noncommutative projective geometry has yet to be developed.  To understand the difficulties in giving such a description, we first review the commutative case.  Let $k$ be a base field.  Associated to a projective $k$-scheme $X$ and a line bundle $\mathcal{L}$ over $X$ generated by $n+1$ global sections which form a $k$-basis for $\Gamma(X,\mathcal{L})$, one gets a map
\begin{equation} \label{eqn.comcomp}
X \longrightarrow \mbox{Proj }\Gamma_{*}(X, \mathcal{L}) \longrightarrow \mbox{Proj }\mathbb{S}(\Gamma(X,\mathcal{L})) = \mathbb{P}^{n}
\end{equation}
Since $\Gamma_{*}(X, \mathcal{L}) := \bigoplus_{n \geq 0}\Gamma(X, \mathcal{L}^{\otimes n})$, this construction depends on a monoidal structure on ${\sf coh }X$, the category of coherent sheaves over $X$.  On the other hand, in the noncommutative world, one often studies categories without an obvious replacement for tensoring over $\mathcal{O}_{X}$, e.g. the category of right modules over a noncommutative ring.  Although there are situations in which this problem can be sidestepped, for example by working with bimodules instead of one-sided modules, it may be desirable to work without this constraint.  One solution to this problem is to replace $\mathcal{L}$ and the associated sequence $(\mathcal{L}^{\otimes i})_{i \in \mathbb{Z}}$ by a sequence of objects $\underline{\mathcal{L}} := (\mathcal{L}_{i})_{i \in \mathbb{Z}}$ in a Hom-finite abelian category ${\sf C}$.  This general perspective is explored in \cite{polish}.  In particular, when $\underline{\mathcal{L}}$ is {\it coherent}, Polishchuk associates to it a functor ${\sf C} \longrightarrow {\sf cohproj }B_{\underline{\mathcal{L}}}$, where $B_{\underline{\mathcal{L}}}$ is the orbit algebra defined by $\underline{\mathcal{L}}$ and ${\sf cohproj }B_{\underline{\mathcal{L}}}$ is the noncommutative analogue of the category of coherent sheaves over $\mbox{Proj }\Gamma_{*}(X, \mathcal{L})$ \cite[Theorem 2.4]{polish}.  This functor is a noncommutative analogue of the first composite of (\ref{eqn.comcomp}).

The purpose of this paper is to describe a noncommutative analogue of the second composite in (\ref{eqn.comcomp}) when the target space is a noncommutative projective line in the sense of \cite{abstractp1}.  The map is constructed in two stages.  First, we define the noncommutative analogue of the map of graded rings $\mathbb{S}(\Gamma(X,\mathcal{L})) \rightarrow \Gamma_{*}(X, \mathcal{L})$ inducing the second composite of (\ref{eqn.comcomp}).  Retaining the notation of the previous paragraph, this has the form of a homomorphism of $\mathbb{Z}$-algebras
\begin{equation} \label{eqn.ringsnc}
\mathbb{S}^{nc}(\underline{\mathcal{L}}) \longrightarrow B_{\underline{\mathcal{L}}},
\end{equation}
where $\mathbb{S}^{nc}(\underline{\mathcal{L}})$ is the noncommutative symmetric algebra of the {\it sequence} $\underline{\mathcal{L}}$, generalizing M. Van den Bergh's noncommutative symmetric algebra of a bimodule \cite{p1bundles}.  At this level of generality, little is known about the structure of $\mathbb{S}^{nc}(\underline{\mathcal{L}})$.  However, since we are ultimately interested in maps to noncommutative projective lines, we would like conditions under which $\mathbb{S}^{nc}(\underline{\mathcal{L}})$ is isomorphic to the noncommutative symmetric algebra of a suitably well behaved bimodule.  Our key insight is that the two concepts of noncommutative symmetric algebra coincide when $\underline{\mathcal{L}}$ is a {\it helix} (Definition \ref{defn.helix}).

The motivation behind the concept of a helix is simple:  If $f:X \rightarrow \mathbb{P}^{1}$ is a morphism of schemes, then $f^{*}$ applied to all twists of the Euler exact sequence
$$
0 \longrightarrow \mathcal{O} \longrightarrow \mathcal{O}(1) \oplus \mathcal{O}(1) \longrightarrow \mathcal{O}(2) \longrightarrow 0
$$
yields, for all $i \in \mathbb{Z}$, an exact sequence of sheaves
\begin{equation} \label{eqn.helixses}
0 \longrightarrow \mathcal{L}_{i} \longrightarrow \mathcal{L}_{i+1} \oplus \mathcal{L}_{i+1} \longrightarrow \mathcal{L}_{i+2} \longrightarrow 0
\end{equation}
over $X$.  A helix is, loosely speaking, a sequence of objects $(\mathcal{L}_{i})_{i \in \mathbb{Z}}$ in a $k$-linear abelian (not necessarily Hom-finite) category ${\sf C}$ satisfying short exact sequences analogous to (\ref{eqn.helixses}).  The notion is a specialization of the notion of helix in the sense of \cite{rudakov}.

In the second stage of the construction of the noncommutative analogue of the second composite of (\ref{eqn.comcomp}), we check that, under suitable conditions on $\underline{\mathcal{L}}$, the map of $\mathbb{Z}$-algebras (\ref{eqn.ringsnc}) induces a functor
\begin{equation} \label{eqn.spacenc}
{\sf Proj } B_{\underline{\mathcal{L}}} \longrightarrow {\sf Proj }\mathbb{S}^{nc}(\underline{\mathcal{L}})
\end{equation}
where, for a $\mathbb{Z}$-algebra $A$, ${\sf Proj }A := {\sf Gr }A/{\sf Tors }A$, ${\sf Gr }A$ is the category of graded right $A$-modules and ${\sf Tors }A$ is the full subcategory of modules, each of whose elements generates a right-bounded submodule.  Some care must be taken to check that, under the appropriate hypotheses on $A$, ${\sf Tors }A$ is a Serre subcategory of ${\sf Gr }A$. Unlike \cite[Theorem 2.4]{polish}, we work without the assumption of coherence, as it is not required for developing the basic machinery of noncommutative projective geometry (see \cite{bondal}).

We conclude the paper by applying our results to construct double covers of Piontkovski's noncommutative projective lines \cite{piont} by noncommutative elliptic curves in the sense of Polishchuk \cite{polish1}.  Recall that J.J. Zhang proved \cite[Theorem 0.1]{zhang} that every connected graded regular algebra of dimension two over $k$ generated in degree one is isomorphic to an algebra of the form \begin{equation} \label{eqn.algebraa}
A=k \langle x_{1}, \ldots, x_{n} \rangle/ (b)
\end{equation}
where $n \geq 2$, and $b=\sum_{i=1}^{n}x_{i}\sigma(x_{n-i})$ for some graded automorphism $\sigma$ of the free algebra.  Piontkovski defines $\mathbb{P}^{1}_{n}$ to be ${\sf cohproj }A$, where $A$ is as above.  He then proves that such categories share many homological properties with the category of coherent sheaves on (commutative) $\mathbb{P}^{1}$.  On the other hand, Polishchuk constructs, for an elliptic curve $X$ over $\mathbb{C}$ and a real number $\theta$, a hereditary abelian category ${\sf C}^{\theta}$ derived equivalent to $X$ and satisfying a form of Serre duality analogous to that of $X$.  We prove the following

\begin{theorem} \label{thm.main}
Let $d \geq 2$ and, for $d>2$, let $\theta_{d} = -\frac{2d}{d-2 + \sqrt{d^2-4}}$.
\begin{enumerate}
\item{} For each $p \in X$, there is a unique helix $\underline{\mathcal{L}}$ in ${\sf coh }X$ with $\mathcal{L}_{-1}=\mathcal{O}_{X}$ and $\mathcal{L}_{0}=\mathcal{O}_{X}(dp)$.  Furthermore, $B_{\underline{\mathcal{L}}}$ is independent of $p$.

\item{} For $d >2$, $B_{\underline{\mathcal{L}}}$ is nonnoetherian and coherent and ${\sf cohproj }B_{\underline{\mathcal{L}}}$ is equivalent to ${\sf C}^{\theta_{d}}$,

\item{} $\mathbb{S}^{nc}(\underline{\mathcal{L}})$ is coherent and ${\sf cohproj }\mathbb{S}^{nc}(\underline{\mathcal{L}})$ is equivalent to $\mathbb{P}^{1}_{d}$, and

\item{} If $d>2$, the functor (\ref{eqn.spacenc}) restricts to a functor ${\sf C}^{\theta_{d}} \longrightarrow \mathbb{P}^{1}_{d}$ which is a double cover in a suitable sense, while if $d=2$, the functor (\ref{eqn.spacenc}) restricts to a functor ${\sf coh }X \rightarrow {\sf coh }\mathbb{P}^{1}$ which is that induced by a double cover $X \rightarrow \mathbb{P}^{1}$.
\end{enumerate}
\end{theorem}
We remark that in the situation of the theorem, the noncommutative analogue of the first composite of (\ref{eqn.comcomp}) does not exist if $d>2$, as the sequence $\underline{\mathcal{L}}$ is not coherent (or even projective in the sense of \cite[Section 2]{polish}) in the category ${\sf coh }X$.

The results in this paper suggest that one avenue towards constructing examples of higher-genus coherent noncommutative curves is through the construction of helices in the category ${\sf coh }Y$, where $Y$ is a genus $g$ curve (the study of helices over $D^{b}(X)$, where $X$ is a projective variety, has been suggested implicitly in \cite{polish1}).  We have not explored this avenue.  In addition, by varying the shape of the exact sequences in the definition of helix, one could extend the techniques of this paper to construct maps to noncommutative projective spaces of larger dimension.  This latter project is a subject of current investigation.

\section{Orbit algebras and noncommutative symmetric algebras}
In this section, after reviewing some basic terminology from the theory of $\mathbb{Z}$-algebras, we define orbit algebras and noncommutative symmetric algebras associated to a sequence of objects $\underline{\mathcal{L}}:=(\mathcal{L}_{i})_{i \in \mathbb{Z}}$ in a $k$-linear abelian category ${\sf C}$.

We first recall (from \cite[Section 2]{quadrics}) the notion of a $\mathbb{Z}$-algebra.  A {\em $\mathbb{Z}$-algebra} $A$ is a $k$-linear pre-additive category whose objects are indexed by $\mathbb{Z}$ and denoted $\{\mathcal{O}_{i}\}_{i \in \mathbb{Z}}$, and whose morphisms are denoted $A_{ij} := \operatorname{Hom}(\mathcal{O}_{-j}, \mathcal{O}_{-i})$.  We will abuse terminology and call the ring $A = \bigoplus_{i,j}A_{ij}$ with multiplication given by composition a $\mathbb{Z}$-algebra. We let $e_{i}$ denote the identity in $A_{ii}$.

Let $A$ be an $\mathbb{Z}$-algebra. A {\em (graded) right $A$-module} is a graded abelian group $M = \bigoplus_{i \in I} M_i$ with multiplication maps $M_i \times A_{ij} \rightarrow M_j$ satisfying the usual module axioms (see \cite{quadrics} for more details). We let ${\sf Gr }A$ denote the category of graded right $A$-modules.  We recall \cite[Section 2]{quadrics} that ${\sf Gr }A$ is a Grothendieck category, hence has enough injectives.

For the remainder of this section, ${\sf C}$ will denote a $k$-linear abelian category and $\underline{\mathcal{L}} := (\mathcal{L}_{i})_{i \in \mathbb{Z}}$ will denote a sequence of objects in ${\sf C}$.

\begin{definition}
The {\it orbit algebra of }$\underline{\mathcal{L}}$, denoted $B_{\underline{\mathcal{L}}}$, is the $\mathbb{Z}$-algebra with
$$
(B_{\underline{\mathcal{L}}})_{ij}=\operatorname{Hom}(\el_{-j}, \el_{-i})
$$
and with multiplication induced by composition.
\end{definition}

\begin{definition}
Let $D_{i} = \operatorname{End} \mathcal{L}_{i}$.  The {\it noncommutative symmetric algebra of} $\underline{\mathcal{L}}$ is the $\mathbb{Z}$-algebra generated in its $(-(i+1), -i)$ component by the $D_{i+1}-D_{i}$-bimodule  $\operatorname{Hom}(\mathcal{L}_{i}, \mathcal{L}_{i+1})$ and with relations generated by the kernels of the compositions
$$
\operatorname{Hom}(\mathcal{L}_{i+1}, \mathcal{L}_{i+2}) \otimes_{D_{i+1}} \operatorname{Hom}(\mathcal{L}_{i}, \mathcal{L}_{i+1}) \rightarrow \operatorname{Hom}(\mathcal{L}_{i}, \mathcal{L}_{i+2}).
$$
We denote it by $\mathbb{S}^{nc}(\underline{\mathcal{L}})$.
\end{definition}

The next example shows that coordinate rings for both noncommutative projective lines and quantum projective planes are noncommutative symmetric algebras of appropriate sequences.
\begin{example}
Suppose $M$ is a bimodule over a pair of division rings which has symmetric duals (see \cite[Definition 2.1.1]{bel} for a definition), and suppose the product of left and right dimensions of $M$ is $\geq 4$.  Let $\mathbb{S}^{nc}(M)$ denote the noncommutative symmetric algebra of $M$ (see \cite[Section 4.1]{p1bundles} or \cite[Definition Section 3.2]{abstractp1} for a definition).  By \cite[Theorem 5.4]{bel}, $\mathbb{S}^{nc}(M)$ is coherent.  Therefore by \cite[Theorem 7.7]{abstractp1}, there is a sequence $\underline{\mathcal{L}}$ in ${\sf cohproj }\mathbb{S}^{nc}(M)$ which is a helix in the sense of \cite{abstractp1}.  Therefore, by Remark \ref{remark.comp} and Proposition \ref{prop.ncsym} in the next section, $\mathbb{S}^{nc}(M) \cong \mathbb{S}^{nc}(\underline{\mathcal{L}})$.

Next, let $C$ denote a smooth elliptic curve over $\mathbb{C}$.  Consider the elliptic helix $(\mathcal{L}_{i})_{i \in \mathbb{Z}}$ in \cite[Section 3.2]{quadrics}.  We construct a new sequence of line bundles over the elliptic curve $C$, $\underline{\mathcal{P}} := (\mathcal{P}_{i})_{i \in \mathbb{Z}}$, as follows:
$$
\mathcal{P}_{i} = \begin{cases} \mathcal{O}_{C} & \mbox{ if $i=-1$} \\
\mathcal{L}_{0} \otimes \mathcal{L}_{-1} \otimes \cdots \otimes \mathcal{L}_{-i} & \mbox{ if $i\geq 0$} \\
\mathcal{L}^{-1}_{1} \otimes \mathcal{L}^{-1}_{2} \otimes \cdots \otimes \mathcal{L}^{-1}_{-i-1} & \mbox{ if $i<-1$}
\end{cases}
$$
Then there is a canonical isomorphism of $\mathbb{Z}$-algebras $\mathbb{S}^{nc}(\underline{\mathcal{P}}) \rightarrow A$, where $A$ is the quadratic three-dimensional $\mathbb{Z}$-algebra corresponding to the triple $(C, \mathcal{L}_{0}, \mathcal{L}_{1})$ \cite[Section 3.2]{quadrics}.
\end{example}

The following result constructs the map (\ref{eqn.ringsnc}) from the introduction.  It's proof follows immediately from the definitions.

\begin{prop} \label{prop.canonmap}
Composition induces a homomorphism of $\mathbb{Z}$-algebras
$$
\mathbb{S}^{nc}(\underline{\mathcal{L}}) \rightarrow B_{\underline{\mathcal{L}}}
$$
which is an isomorphism in degrees zero and one.
\end{prop}

We next observe that the map in Proposition \ref{prop.canonmap} satisfies several natural compatibilities.
\begin{lemma} \label{lemma.functorringmap}
Suppose ${\sf C}$ is a $k$-linear abelian category and $\underline{\mathcal{L}}$ is a sequence in ${\sf C}$.

\begin{enumerate}

\item{}  If ${\sf D}$ is a $k$-linear abelian category, $F:{\sf C} \rightarrow {\sf D}$ is a $k$-linear functor and $F(\underline{\mathcal{L}})$ denotes the sequence $(F(\mathcal{L}_{i}))_{i \in \mathbb{Z}}$ in ${\sf D}$, then $F$ induces $\mathbb{Z}$-algebra homomorphisms
$$
B_{\underline{\mathcal{L}}} \longrightarrow B_{F(\underline{\mathcal{L}})}
$$
and
$$
\mathbb{S}^{nc}(\underline{\mathcal{L}}) \longrightarrow \mathbb{S}^{nc}(F(\underline{\mathcal{L}}))
$$
which are isomorphisms if $F$ is an equivalence.  Furthermore, these maps are compatible with the canonical map from Proposition \ref{prop.canonmap}.

\item{} If $\underline{\mathcal{M}}$ is a sequence in ${\sf C}$ and there are isomorphisms $\phi_{i}:\mathcal{L}_{i} \longrightarrow \mathcal{M}_{i}$ for all $i \in \mathbb{Z}$, then the induced isomorphisms
    $$
    \operatorname{Hom}(\mathcal{L}_{i},\mathcal{L}_{j}) \longrightarrow \operatorname{Hom}(\mathcal{M}_{i},\mathcal{M}_{j})
    $$
give isomorphisms of $\mathbb{Z}$-algebras
$$
B_{\underline{\mathcal{L}}} \longrightarrow B_{\underline{\mathcal{M}}}
$$
and
$$
\mathbb{S}^{nc}(\underline{\mathcal{L}}) \longrightarrow \mathbb{S}^{nc}(\underline{\mathcal{M}})
$$
compatible with the map from Proposition \ref{prop.canonmap}.
\end{enumerate}
\end{lemma}

We introduce some terminology which will be used for the remainder of this section.  Recall, from \cite{quadrics}, that if $A$ is a $\mathbb{Z}$-algebra and $i \in \mathbb{Z}$, then $A(i)$ is the canonical $\mathbb{Z}$-algebra with $A(i)_{jk}=A_{j+i, k+i}$.  If $A \cong A(i)$ for some $i$, then $A$ is called {\it $i$-periodic}.  If $C$ is a $\mathbb{Z}$-graded algebra, we write $\overline{C}$ for the canonical $\mathbb{Z}$-algebra with $\overline{C}_{ij} := C_{j-i}$.  By \cite[Lemma 2.4]{quadrics}, if $A$ is a 1-periodic $\mathbb{Z}$-algebra, then there exists a $\mathbb{Z}$-graded algebra $\widehat{A}$ such that $A \cong \overline{\widehat{A}}$, and in this case there is an equivalence ${\sf Gr }A \longrightarrow {\sf Gr }\widehat{A}$.

\begin{lemma} \label{lemma.oneperiod}
Suppose $A$ and $B$ are 1-periodic $\mathbb{Z}$-algebras via isomorphisms $\theta_{A}:A \longrightarrow A(1)$ and $\theta_{B}:B \longrightarrow B(1)$.  If $\phi:A \longrightarrow B$ is a homomorphism such that the diagram
$$
\begin{CD}
A & \overset{\phi}{\longrightarrow} & B \\
@V{\theta_{A}}VV @VV{\theta_{B}}V \\
A(1) & \underset{\phi(1)}{\longrightarrow} & B(1)
\end{CD}
$$
commutes, then the induced map of $\mathbb{Z}$-graded algebras $\widehat{\phi}:\widehat{A} \longrightarrow \widehat{B}$ is a homomorphism.
\end{lemma}

We conclude the section by showing that the canonical map in Proposition \ref{prop.canonmap} agrees with the usual one in the commutative case.
\begin{example} \label{example.commagree}
Let $X$ be a noetherian $k$-scheme, let $\mathcal{L}$ denote an invertible sheaf over $X$ and let $\underline{\mathcal{L}}$ denote the sequence $(\mathcal{L}^{\otimes i})_{i \in \mathbb{Z}}$.  Let $F:{\sf coh }X \longrightarrow {\sf coh }X$ denote the equivalence $-\otimes_{\mathcal{O}_{X}}\mathcal{L}^{-1}$.  Then $B_{F(\underline{\mathcal{L}})} \cong B_{\underline{\mathcal{L}}}(1)$ and $\mathbb{S}^{nc}(F(\underline{\mathcal{L}})) \cong \mathbb{S}^{nc}(\underline{\mathcal{L}})(1)$ by Lemma \ref{lemma.functorringmap}(2) so that
$B_{\underline{\mathcal{L}}}$ and $\mathbb{S}^{nc}(\underline{\mathcal{L}})$ are $1$-periodic.  Furthermore, the induced diagram
$$
\begin{CD}
\mathbb{S}^{nc}(\underline{\mathcal{L}}) & \longrightarrow & B_{\underline{\mathcal{L}}} \\
@VVV @VVV \\
\mathbb{S}^{nc}(\underline{\mathcal{L}})(1) & \longrightarrow & B_{\underline{\mathcal{L}}}(1)
\end{CD}
$$
with horizontal maps from Proposition \ref{prop.canonmap}, commutes by Lemma \ref{lemma.functorringmap}.  Therefore, by Lemma \ref{lemma.oneperiod}, the induced map of $\mathbb{Z}$-graded algebras,
\begin{equation} \label{eqn.hateqn}
\widehat{\mathbb{S}^{nc}(\underline{\mathcal{L}})} \longrightarrow \widehat{B_{\underline{\mathcal{L}}}}
\end{equation}
is a homomorphism.  Finally, it is straightforward to check that $\widehat{\mathbb{S}^{nc}(\underline{\mathcal{L}})} \cong \mathbb{S}(\Gamma(X,\mathcal{L}))$ and $\widehat{B_{\underline{\mathcal{L}}}} \cong \Gamma_{*}(X,\mathcal{L})$, and that under these identifications, the homomorphism (\ref{eqn.hateqn}) induces the second composite of (\ref{eqn.comcomp}).
\end{example}

\section{Helices}
Throughout this section, ${\sf C}$ will denote a $k$-linear abelian category and $\underline{\mathcal{L}} := (\mathcal{L}_{i})_{i \in \mathbb{Z}}$ will denote a sequence of objects in ${\sf C}$.  The purpose of this section is to relate $\mathbb{S}^{nc}(\underline{\mathcal{L}})$ to the noncommutative symmetric algebra of a bimodule.  To this end, we recall that if $\mathcal{L}$ and $\mathcal{M}$ are objects of ${\sf C}$ and $\operatorname{End }\mathcal{L} = :D$, then one can define an object ${}^{*}\operatorname{Hom}(\mathcal{M}, \mathcal{L}) \otimes_{D} \mathcal{L}$ of ${\sf C}$ as in \cite[Section B3]{az2}.  In particular, if $D$ is a noetherian $k$-algebra and $\operatorname{Hom}(\mathcal{M}, \mathcal{L})$ is a finitely generated projective left $D$-module, then there is a canonical map
$$
\eta_{\mathcal{M}}: \mathcal{M} \rightarrow {}^{*}\operatorname{Hom}(\mathcal{M}, \mathcal{L}) \otimes_{D} \mathcal{L}.
$$

\begin{definition} \label{defn.helix}
We say the sequence $\underline{\mathcal{L}}$ is a {\it helix} if, for all $i \in \mathbb{Z}$,
\begin{enumerate}
\item{} $\operatorname{End }\mathcal{L}_{i}= :D_{i}$ is a division ring with $k$ in its center,

\item{} $\operatorname{Hom}(\mathcal{L}_{i}, \mathcal{L}_{i+1})$ is finite-dimensional as both a left $D_{i+1}$-module and a right $D_{i}$-module.

\item{} There is a short exact sequence, which we call the {\it $i$th Euler sequence},
$$
0 \rightarrow \mathcal{L}_{i} \overset{\eta_{\mathcal{L}_{i}}}{\rightarrow} {}^{*}\operatorname{Hom}(\mathcal{L}_{i}, \mathcal{L}_{i+1}) \otimes_{D_{i+1}} \mathcal{L}_{i+1} \rightarrow \mathcal{L}_{i+2} \rightarrow 0.
$$

\item{} The canonical one-to-one $k$-algebra map $\Phi_{i}:D_{i} \rightarrow D_{i+2}$ (see \cite[Lemma 4.5]{abstractp1}) is an isomorphism.

\item{} If, for each $i$, we endow $\operatorname{Hom}(\el_{i+1},\el_{i+2})$ with the $D_{i} - D_{i+1}$-bimodule structure from (4), the canonical map ${}^{*}\operatorname{Hom}(\el_{i},\el_{i+1}) \longrightarrow \operatorname{Hom}(\el_{i+1},\el_{i+2})$ constructed by applying $\operatorname{Hom}(\mathcal{L}_{i+1},-)$ to the $i$th Euler sequence is an isomorphism of bimodules.
\end{enumerate}
\end{definition}

\begin{example} \label{example.helixcomm}
Let $X$ be a projective variety over $k$ with structure sheaf $\mathcal{O}$, and suppose $\mathcal{L}$ is line bundle on $X$ such that $\mathcal{L}$ is generated by two global sections, and $h^{0}(\mathcal{L}^{-1})=0$.  We prove that if $\mathcal{L}_{i}:=\mathcal{L}^{\otimes i}$, then $\underline{\mathcal{L}}:=\{\mathcal{L}_{i}\}$ is a helix.

To prove this, we note that conditions (1), (2) and (4) are clear.  To prove (3), we must show exactness of
\begin{equation} \label{equation.ses}
0 \rightarrow \mathcal{L}_{i} \overset{\eta_{\mathcal{L}_{i}}}{\rightarrow} {}^{*}\operatorname{Hom}(\mathcal{L}_{i}, \mathcal{L}_{i+1}) \otimes_{k} \mathcal{L}_{i+1} \rightarrow \mathcal{L}_{i+2} \rightarrow 0.
\end{equation}
We will construct the sequence when $i=0$, and obtain the others by tensoring by the approriate power of $\mathcal{L}$.  We let $\{-s_{0}, s_{1}\}$ be a linearly independent set of global sections of $\mathcal{L}$ which generate $\mathcal{L}$.  We let $s: \mathcal{L}^{\oplus 2} \rightarrow \mathcal{L}^{\otimes 2}$ denote the tensor product of the epimorphism $(s_{1},-s_{0}): \mathcal{O} \oplus \mathcal{O} \rightarrow \mathcal{L}$ with $\mathcal{L}$.  It is then routine to check that, since $X$ is a variety, the sequence
$$
0 \rightarrow \mathcal{O} \overset{(s_{0},s_{1})}{\longrightarrow} \mathcal{L} \oplus \mathcal{L} \overset{s}{\longrightarrow} \mathcal{L}^{\otimes 2} \rightarrow 0
$$
is exact, whence (3) follows.

Finally, to prove condition (5), we apply $\operatorname{Hom}(\mathcal{L}_{i+1}, -)$ to (\ref{equation.ses}) to obtain an exact sequence
$$
0 \rightarrow \operatorname{Hom}(\mathcal{L}_{i+1},\mathcal{L}_{i}) \rightarrow {}^{*}\operatorname{Hom}(\mathcal{L}_{i}, \mathcal{L}_{i+1}) \rightarrow \operatorname{Hom}(\mathcal{L}_{i+1},\mathcal{L}_{i+2}).
$$
The left term is $0$ by hypothesis.  Thus, the right map is an inclusion.  Since the dimension of the middle term equals that of the right term, condition (5) holds.
\end{example}

\begin{example} \label{example.p1}
Suppose ${\sf C}={\sf coh }\mathbb{P}^{1}$.  We write $\mathcal{O}$ for $\mathcal{O}_{\mathbb{P}^{1}}$.  By Example \ref{example.helixcomm}, the canonical maps
$$
\mathcal{O} \rightarrow {}^{*}\operatorname{Hom}(\mathcal{O},\mathcal{O}(1)) \otimes \mathcal{O}(1),
$$
and
$$
\operatorname{Hom}(\mathcal{O}, \mathcal{O}(1)) \otimes \mathcal{O} \rightarrow \mathcal{O}(1)
$$
extend uniquely to the right and left to give a helix, that is, if we let $\mathcal{L}_{0} := \mathcal{O}$ and $\mathcal{L}_{1} := \mathcal{O}(1)$, then we can let $\mathcal{L}_{2}$ be the cokernel of the first map and $\mathcal{L}_{-1}$ be the kernel of the second map, etc., and we obtain a helix (whose associated short exact sequences consist of all twists of the Euler sequence).

We now show, in contrast, that for $n \geq 2$, the canonical maps
\begin{equation} \label{eqn.canonn}
\mathcal{O} \rightarrow {}^{*}\operatorname{Hom}(\mathcal{O},\mathcal{O}(n)) \otimes_{k} \mathcal{O}(n),
\end{equation}
and
\begin{equation} \label{eqn.cocanonn}
\operatorname{Hom}(\mathcal{O}, \mathcal{O}(n)) \otimes \mathcal{O} \rightarrow \mathcal{O}(n)
\end{equation}
do not extend to a helix.

To prove this, suppose the cokernel of (\ref{eqn.canonn}) had the form $\sum_{i=1}^{n} \mathcal{O}(a_{i}) \oplus \mathcal{T}$ where $\mathcal{T}$ is torsion.  Twisting the resulting short exact sequence by $-n$ and computing global sections yields:
\begin{equation} \label{eqn.minusn}
\sum_{i=1}^{n}h^{0}(\mathcal{O}(a_{i}-n))+h^{0}(\mathcal{T})=2n.
\end{equation}
Repeating this for $-(n+1)$ yields
\begin{equation} \label{eqn.minusn1}
\sum_{i=1}^{n}h^{0}(\mathcal{O}(a_{i}-(n+1)))+h^{0}(\mathcal{T})=n.
\end{equation}
Finally, if the maps (\ref{eqn.canonn}) and (\ref{eqn.cocanonn}) induced a helix, we would also have $$
{}^{*}\operatorname{Hom}(\mathcal{O}, \mathcal{O}(n)) \cong \operatorname{Hom}(\mathcal{O}(n), \sum_{i=1}^{n}\mathcal{O}(a_{i})\oplus \mathcal{T}),
$$
which would imply that
\begin{equation} \label{eqn.helix}
\sum_{i=1}^{n}h^{0}(\mathcal{O}(a_{i}-n))+h^{0}(\mathcal{T})=n+1.
\end{equation}
Now, subtracting (\ref{eqn.minusn1}) from (\ref{eqn.helix}) implies that
$$
1= \sum_{i=1}^{n} h^{0}(\mathcal{O}(a_{i}-n))-h^{0}(\mathcal{O}(a_{i}-(n+1)))
$$
which implies that $a_{i} \geq n$ for exactly one $i$.  Then (\ref{eqn.helix}) implies that $h^{0}(\mathcal{T})=n$.  It thus follows that (\ref{eqn.minusn}) is impossible since $n>1$ and the result follows.
\end{example}

\begin{remark} \label{remark.comp}
A helix in the sense of \cite{abstractp1} is a helix in the sense of this paper.  For, the conditions (1), (3), (4) and (5) of helix from \cite{abstractp1} implies conditions (1), (2), (3) and (4) of the notion of helix studied here.  Condition (2) of a helix from \cite{abstractp1} implies, by \cite[Proposition 5.1]{abstractp1}, that condition (5) of the current notion of helix holds.

On the other hand, if $X$ is a smooth elliptic curve over a field and $\mathcal{L}$ is a degree two line bundle, then  $\{\mathcal{L}^{\otimes i}\}$ is a helix over ${\sf coh }X$ by Example \ref{example.helixcomm}, but doesn't satisfy condition (2) of helix from \cite{abstractp1}.

 A notion of {\it geometric helix} was introduced in both \cite{moriueyama} and \cite{bondalp}, (the latter agreeing with \cite{rudakov}).  Although these notions differ in general (see \cite[Remark 3.17]{moriueyama}), they agree with each other and with the notion of helix from \cite{abstractp1} in certain circumstances (details will appear in another paper).  The example in the previous paragraph illustrates the fact that there exist helices in the sense of this paper which are not geometric helices in the sense of  \cite{moriueyama} or \cite{bondalp}.
\end{remark}

\begin{lemma} \label{lemma.functorringmap}
Suppose ${\sf C}$ and ${\sf D}$ are $k$-linear abelian categories, $F:{\sf C} \rightarrow {\sf D}$ is a $k$-linear equivalence and $\underline{\mathcal{L}}$ is a helix in ${\sf C}$.  Let $F(\underline{\mathcal{L}})$ denote the sequence $(F(\mathcal{L}_{i}))_{i \in \mathbb{Z}}$ in ${\sf D}$.  Then $F(\underline{\mathcal{L}})$ is a helix in ${\sf D}$.
\end{lemma}

\begin{proof}
Properties (1), (2) and (3) in the definition of helix are easy to check for  $F(\underline{\mathcal{L}})$.  For property (4), we note that the canonical map $\operatorname{End}(F(\mathcal{L}_{i})) \rightarrow \operatorname{End}(F(\mathcal{L}_{i+2}))$ sends $F(\psi) \in \operatorname{End}(F(\mathcal{L}_{i}))$ to $F(\Phi_{i}(\psi))$, and so is an isomorphism.  Finally, we note that in applying $(\mathcal{L}_{i+1},-)$ to the $i$th Euler sequence, property (5) of a helix for $\underline{\mathcal{L}}$ implies that the corresponding first connecting homomorphism is zero.  Therefore, the first connecting homomorphism corresponding to the application of $(F(\mathcal{L}_{i+1}),-)$ to the $i$th Euler exact sequence for $F(\underline{\mathcal{L}})$ must be zero as well, whence the result.
\end{proof}

\begin{prop} \label{prop.ncsym}
Suppose $\underline{\mathcal{L}}$ is a helix.  Then, for all $i \in \mathbb{Z}$,
\begin{enumerate}
\item{} the $\mathbb{Z}$-algebra $\mathbb{S}^{nc}(\operatorname{Hom}(\mathcal{L}_{i}, \mathcal{L}_{i+1}))$ exists, and

\item{} there is an isomorphism of $\mathbb{Z}$-algebras
$$
\mathbb{S}^{nc}(\operatorname{Hom}(\mathcal{L}_{-1}, \mathcal{L}_{0})) \longrightarrow \mathbb{S}^{nc}(\underline{\mathcal{L}}).
$$
\end{enumerate}
\end{prop}

\begin{proof}
To prove the first result, it suffices to show the $D_{i+1}-D_{i}$-bimodule $\operatorname{Hom}(\mathcal{L}_{i}, \mathcal{L}_{i+1})$ is admissible (see \cite{abstractp1} for a definition). The fact that $\underline{\mathcal{L}}$ is a helix implies that the proof of \cite[Proposition 5.3]{abstractp1} holds, establishing the result.

To prove the second result, we note that, by property (5) of helix, there is a canonical isomorphism from the free $\mathbb{Z}$-algebra generated by the sequence of bimodules $\{\operatorname{Hom}(\mathcal{L}_{-1}, \mathcal{L}_{0})^{i*}\}_{i \in \mathbb{Z}}$ to the free $\mathbb{Z}$-algebra generated by the sequence of bimodules $\{\operatorname{Hom}(\mathcal{L}_{-i-1}, \mathcal{L}_{-i})\}_{i \in \mathbb{Z}}$.  By the proof of \cite[Lemma 5.2]{abstractp1}, this isomorphism descends to a homomorphism
$$
\mathbb{S}^{nc}(\operatorname{Hom}(\mathcal{L}_{-1}, \mathcal{L}_{0})) \longrightarrow \mathbb{S}^{nc}(\underline{\mathcal{L}})
$$
which is an isomorphism in degrees zero and one.  Furthermore, by applying $(\mathcal{L}_{i},-)$ to the $i$th Euler sequence, we see that the isomorphism of free $\mathbb{Z}$-algebras maps the relations in degree two in $\mathbb{S}^{nc}(\operatorname{Hom}(\mathcal{L}_{-1}, \mathcal{L}_{0}))$ onto the relations in degree two in $\mathbb{S}^{nc}(\underline{\mathcal{L}})$.  The result follows from this.
\end{proof}

Combining Proposition \ref{prop.canonmap} and Proposition \ref{prop.ncsym}, we obtain the following
\begin{cor} \label{cor.main}
Suppose $\underline{\mathcal{L}}$ is a helix.  Then there is a morphism of $\mathbb{Z}$-algebras
$$
\mathbb{S}^{nc}(\operatorname{Hom}(\mathcal{L}_{-1}, \mathcal{L}_{0})) \longrightarrow B_{\underline{\mathcal{L}}}
$$
which is an isomorphism in degrees zero and one.
\end{cor}

\section{The main theorem}
In this section, we prove our main result (Theorem \ref{thm.makemap}), which gives hypotheses under which the map of $\mathbb{Z}$-algebras from Corollary \ref{cor.main} induces a map between spaces.

\begin{prop} \label{prop.extzero}
Suppose $\underline{\mathcal{L}}$ is a helix over ${\sf C}$ such that there exists an $n \geq 0$ with the property that for all $l \geq n$, $\operatorname{Ext}^{1}(\mathcal{L}_{i}, \mathcal{L}_{i+l})=0$ for all $i \in \mathbb{Z}$.  Then, for all $i \in \mathbb{Z}$, the right $B:=B_{\underline{\mathcal{L}}}$-module $e_{i}B_{\geq i+1}$ is generated by $B_{i,i+1}, \ldots, B_{i, i+n+1}$.
\end{prop}

\begin{proof}
Suppose $l \geq n$.  Applying $\operatorname{Hom}(\mathcal{L}_{j}, -)$ to the $j+l$th Euler sequence and using property (5) of a helix, we deduce that multiplication
$$
B_{-(j+l+2), -(j+l+1)} \otimes B_{-(j+l+1), -j} \rightarrow B_{-(j+l+2), -j}
$$
is surjective.  Setting $i=-(j+l+2)$, we find that multiplication
$$
B_{i,i+1} \otimes B_{i+1, i+l+2} \rightarrow B_{i, i+l+2}
$$
is surjective for $l \geq n$.  The result follows.
\end{proof}

We now recall some terminology from \cite[Section 2]{quadrics}.  We say that a $\mathbb{Z}$-algebra $A$ is {\it positively graded} if $A_{ij}=0$ for all $i>j$.  In what follows, we will abuse terminology by saying `$A$ is a $\mathbb{Z}$-algebra' if $A$ is a positively graded $\mathbb{Z}$-algebra.  We say a $\mathbb{Z}$-algebra $A$ is {\it connected} if for all $i$, $A_{ii}$ is a division ring over $k$ and $A_{i,i+1}$ is finite-dimensional over both $A_{ii}$ and $A_{i+1,i+1}$.

Suppose $A$ is a $\mathbb{Z}$-algebra.  A graded right $A$-module $M$ is {\it right bounded} if $M_{n}=0$ for all $n>>0$.  We let ${\sf Tors }A$ denote the full subcategory of ${\sf Gr }A$ consisting of modules whose elements $m$ have the property that the right $A$-module generated by $m$ is right bounded.  It follows from \cite[Lemma 3.5]{morinyman} that if $A$ is connected and $e_{j}A_{\geq j+1}$ is finitely generated for all $j \in \mathbb{Z}$, then ${\sf Tors }A$ is a localizing subcategory of ${\sf Gr }A$.  Furthermore, since ${\sf Gr }A$ has enough injectives,  if $\pi:{\sf Gr }A \rightarrow {\sf Gr }A/{\sf Tors }A = : {\sf Proj }A$ is the quotient functor, then there exists a right adjoint to $\pi$.

\begin{theorem} \label{thm.makemap}
Suppose that $\underline{\mathcal{L}}$ is a helix over ${\sf C}$ such that
\begin{itemize}
\item{} $\operatorname{Hom}(\mathcal{L}_{j}, \mathcal{L}_{i})=0$ for all $i, j \in \mathbb{Z}$ with $j>i$,

\item{} there exists an $n \geq 0$ with the property that for all $l \geq n$, $\operatorname{Ext}^{1}(\mathcal{L}_{i}, \mathcal{L}_{i+l})=0$ for all $i \in \mathbb{Z}$, and

\item{} for all $i \in \mathbb{Z}$, $B_{i,j}$ is finite-dimensional over $B_{jj}$ for $i+1 \leq j \leq i+n+1$.
\end{itemize}

Then
\begin{enumerate}
\item{} $B_{\underline{\mathcal{L}}}$ is connected,

\item{} the subcategory ${\sf Tors }B_{\underline{\mathcal{L}}}$ is a localizing subcategory of ${\sf Gr }B_{\underline{\mathcal{L}}}$,

\item{} the restriction functor ${\sf Gr }B_{\underline{\mathcal{L}}} \rightarrow {\sf Gr }\mathbb{S}^{nc}(\operatorname{Hom }(\mathcal{L}_{-1}, \mathcal{L}_{0}))$ induced by the map from Corollary \ref{cor.main} induces a functor
\begin{equation} \label{eqn.geomap}
{\sf Proj }B_{\underline{\mathcal{L}}} \rightarrow {\sf Proj }\mathbb{S}^{nc}(\operatorname{Hom }(\mathcal{L}_{-1}, \mathcal{L}_{0})).
\end{equation}

\end{enumerate}
\end{theorem}

\begin{proof}
The first assertion is immediate from the first hypothesis and properties (1) and (2) of a helix.

Next we note that, by the comment preceding the statement of the theorem, it suffices, by (1), to prove that $e_{i}B_{\geq i+1}$ is finitely generated for all $i \in \mathbb{Z}$.  By the second hypothesis, Proposition \ref{prop.extzero} implies that $e_{i}B_{\geq i+1}$ is generated by
$$
B_{i,i+1}, \ldots, B_{i, i+n+1}
$$
and each of these is finite-dimensional on the right by the third hypothesis.

To prove (3), it is immediate that the restriction functor sends torsion modules to torsion modules.  Therefore, it suffices, \cite[Lemma 1.1]{smithcor}, to show that ${\sf Tors }\mathbb{S}^{nc}(\operatorname{Hom }(\mathcal{L}_{-1}, \mathcal{L}_{0}))$ is localizing.  This follows from the comment preceding the statement of the theorem, since $e_{j}\mathbb{S}^{nc}(\operatorname{Hom }(\mathcal{L}_{-1}, \mathcal{L}_{0}))_{\geq j+1}$ is generated by $\operatorname{dim}_{D_{-j-1}}\operatorname{Hom}(\mathcal{L}_{-j-1}, \mathcal{L}_{-j})$ elements, which is finite since $\underline{\mathcal{L}}$ is a helix.
\end{proof}

\begin{remark} \label{remark.extfinite}
We note that the third hypothesis in Theorem \ref{thm.makemap} is automatic in case ${\sf C}$ is Hom-finite, or in case $B$ is right Ext-finite (see \cite[Remark 3.3]{morinyman}).
\end{remark}





\begin{example} \label{example.helixcomm2}
Let $X$ be a projective variety over $k$ with structure sheaf $\mathcal{O}$, and suppose $\mathcal{L}$ is line bundle on $X$ such that $\mathcal{L}$ is generated by two global sections, and $h^{0}(\mathcal{L}^{-1})=0$.  By Example \ref{example.helixcomm}, we know that $(\mathcal{L}^{\otimes i})_{i \in \mathbb{Z}}$ is a helix.  We show that the other hypotheses of Theorem \ref{thm.makemap} hold as well.  The first hypothesis follows from the fact that $\Gamma(X,\mathcal{L}^{-1})=0$ and property (3) of a helix (see \cite[Lemma 4.1]{abstractp1}).  The second hypothesis follows from Serre vanishing on $X$, and the third hypothesis follows from Hom-finiteness on ${\sf coh }X$.
\end{example}

\section{An application: double covers of $\mathbb{P}^{1}_{n}$}
The goal of this section is to prove Theorem \ref{thm.main}.  Let $X$ denote a smooth elliptic curve over $k= \mathbb{C}$.  We begin by recalling terminology related to the notion of a simple pair of coherent sheaves over $X$ from \cite[Definition 1]{rudakov2}.  A coherent sheaf, $\mathcal{E}$, over $X$ is {\em simple} if $\operatorname{End }(\mathcal{E})=k$.  A pair of simple coherent sheaves, $\cL_0, \cL_1$, is called a {\em simple pair}, denoted $(\cL_0, \cL_1)$, if $\operatorname{Ext}^{i}(\cL_0, \cL_1) \neq 0$ for only one value of $i$.  A simple pair $(\cL_0,\cL_1)$ is an {\em injective pair} if $\operatorname{Ext}^1(\cL_0,\cL_1) = 0$ and every non-zero morphism $f \colon \cL_0 \arr \cL_1$ is injective.  A simple (resp. injective) pair $(\cL_0,\cL_1)$ such that $\cL_0$ and $\cL_1$ are vector bundles is called a {\em simple pair of bundles} (resp. {\em injective pair of bundles}).

\begin{lemma}  \label{lem:EllipticRightMutate}
Let $(\cL_0,\cL_1)$ be an injective pair of bundles. Suppose that $n:= \dim_{k} \Hom (\cL_0,\cL_1) > 1$. Then the canonical short exact sequence
$$ 0 \arr \cL_0 \xrightarrow{\iota} \operatorname{Hom}(\cL_0,\cL_1)^* \otimes \cL_1 \xrightarrow{\pi} \cL_2 \arr 0 $$
defines an injective pair of bundles $(\cL_1,\cL_2)$.
\end{lemma}

\begin{proof}
We know from \cite[Lemma~3]{rudakov2} that $(\cL_1,\cL_2)$ is a simple pair and that the morphism $\pi$ above is the canonical one induced by
$$ \operatorname{Hom}(\cL_1,\cL_2) \otimes \cL_1 \arr \cL_2.$$
Suppose there is some non-zero $f \colon \cL_1 \arr \cL_2$ with non-zero kernel $\mathcal{K}$. We may thus write  $\operatorname{Hom}(\cL_1,\cL_2) \otimes \cL_1 = \cL_1^{n-1} \oplus \cL_1$ (an internal direct sum) so that $\pi|_{\cL_1} = f$. We thus obtain a natural commutative square
$$\begin{CD}
\cL_0 @>{\iota}>> \cL_1^{n-1} \oplus \cL_1 \\
@VVV @VVV \\
\cL_0/\iota^{-1}(\mathcal{K}) @>>> \cL_1^{n-1} \oplus \cL_1/\mathcal{K}
\end{CD}.$$
Now $n-1 >0$, so there exists a non-zero morphism $\cL_0 \arr \cL_1$ with kernel $\iota^{-1}(\mathcal{K})$. This contradiction shows that in fact $(\cL_1,\cL_2)$ is an injective pair.

Injectivity and the assumption $n>1$ ensures that the rank of $\cL_1^n$ exceeds that of $\cL_0$. Hence, $\text{rank}\, \cL_2 >0$. But $\cL_2$ is simple and hence indecomposable. It thus follows that it is also a vector bundle. Finally, $\operatorname{Hom}(\cL_1,\cL_2) \cong \operatorname{Hom}(\cL_0,\cL_1)^* \neq 0$ so simplicity ensures $\operatorname{Ext}^1(\cL_1,\cL_2) = 0$.
\end{proof}
The hypothesis in Lemma \ref{lem:EllipticRightMutate} that $n>1$ is needed, for given $p \in X$ we have a canonical short exact sequence
\begin{equation} \label{eq:Op}
 0 \arr \cO \arr \cO(p) \arr \cO_p \arr 0
\end{equation}
and the simple pair $(\cO(p),\cO_p)$ is not an injective pair.

Suppose we are in the situation of \cite[Lemma~3(1)]{rudakov2}, that is, there is a simple pair $(\cL_0,\cL_1)$ with $\operatorname{Ext}^1(\cL_0,\cL_1) = 0$ and a short exact sequence
\begin{equation} \label{eq:HomMutate} 0 \arr \cL_0 \xrightarrow{\iota} V \otimes \cL_1 \xrightarrow{\pi} \cL_2 \arr 0
\end{equation}
where $\iota, \pi$ are canonical maps associated with isomorphisms  $V \cong \Hom(\cL_0,\cL_1)^* \cong \Hom(\cL_1,\cL_2)$. We shall call such a sequence a {\em Hom-mutation sequence}.


\begin{prop}  \label{prop:DualMutate}
Suppose that $(\mathcal{L}_{0}, \mathcal{L}_{1})$ is a simple pair of bundles and $\cL_2$ is also a bundle. Then the dual exact sequence to (\ref{eq:HomMutate})
$$  0 \arr \cL_2^* \xrightarrow{\pi^*} V^* \otimes \cL_1^* \xrightarrow{\iota^*} \cL_0^* \arr 0$$
is also a Hom-mutation sequence.
\end{prop}
\begin{proof}
The sequence is exact since we assumed that $\cL_0, \cL_1$, and $\cL_2$ are bundles, and $(\cL_1^*,\cL_0^*)$ is a simple pair since $\operatorname{Ext}^i(\cL_1^*,\cL_0^*) \cong H^i(X, \cL_0^* \otimes \cL_1) \cong \operatorname{Ext}^i(\cL_0,\cL_1)$. It remains only to observe that $\iota^*$ is the canonical map obtained by identifying $V^* \cong \operatorname{Hom}(\cL_0,\cL_1)  \cong \operatorname{Hom}(\cL_1^*, \cL_0^*)$.
\end{proof}

\begin{thm}  \label{thm:ExistHelix}
Let $(\cL_0,\cL_1)$ be an injective pair of bundles such that
$$
\dim_{k} \operatorname{Hom}(\cL_0,\cL_1) >1.
$$
Suppose that the simple pair $(\cL_1^*,\cL_0^*)$ is also an injective pair of bundles. Then, up to isomorphism, the pair extends to a unique helix $\ldots, \cL_{-1},\cL_0, \cL_1, \cL_2, \ldots $.
\end{thm}
\begin{proof}
Lemma~\ref{lem:EllipticRightMutate} and induction shows that we can extend to a helix $\{ \cL_i \}_{i \geq 0}$ on the right, i.e. we have Hom-mutation sequences of the form
$$  0 \arr \cL_i \xrightarrow{\iota} V_i \otimes \cL_{i+1} \xrightarrow{\pi} \cL_{i+2} \arr 0$$
for all $i \geq 0$. The same reasoning applied to $(\cL_1^*,\cL_0^*)$ shows we have Hom-mutation sequences of the form
$$  0 \arr \cL^*_{i+2} \arr V^*_i \otimes \cL^*_{i+1} \arr \cL^*_i \arr 0 $$
for $i < 0$. Proposition~\ref{prop:DualMutate} shows that the duals of these Hom-mutation sequences are still Hom-mutation sequences. We thus obtain a full helix $(\cL_i)_{i \in \mathbb{Z}}$.
\end{proof}
For the remainder of the paper, we suppose $\cL_{-1},\cL_0$ are line bundles with $d:= \deg \cL_0 - \deg \cL_{-1} >1$.  By Theorem \ref{thm:ExistHelix}, $\cL_{-1},\cL_0$ determine a unique helix $\underline{\mathcal{L}}$.  We let $r_m$ denote the rank of $\mathcal{L}_{m}$, we let $d_m$ be the degree of $\mathcal{L}_{m}$, and we let $\mu_m = d_m/r_m$.

\begin{lemma} \label{lemma.slopeinc}
For all $m \in \mathbb{Z}$, $\mu_m < \mu_{m+1}$.
\end{lemma}

\begin{proof}
First, note that $\mu_0-\mu_{-1}=d$.  Next, suppose $\mu_{m-2}<\mu_{m-1}$.  Since rank and degree are additive on short exact sequences of vector bundles, $d_m=nd_{m-1}-d_{m-2}$ and $r_m=nr_{m-1}-r_{m-2}$.  We note that $d_mr_{m-1}=(nd_{m-1}-d_{m-2})r_{m-1}$ which, by our supposition, is greater than $nd_{m-1}r_{m-1}-d_{m-1}r_{m-2}=d_{m-1}r_{m-2}$.

The proof that if $\mu_{m-1}<\mu_{m}$ then $\mu_{m-2}<\mu_{m-1}$ is similar.
\end{proof}

\begin{lemma} \label{lemma.ext}
Suppose $i, j \in \mathbb{Z}$ and $i<j$.  Then $\operatorname{Ext}^{1}(\mathcal{L}_{i},\mathcal{L}_{j})=0$.
\end{lemma}

\begin{proof}
By Serre duality, it is enough to prove that $\operatorname{Hom}(\mathcal{L}_{j}, \mathcal{L}_{i})=0$, and to prove this, it suffices to prove $\operatorname{Hom }(\mathcal{O}, \mathcal{L}_{i} \otimes \mathcal{L}_{j}^{*})=0$.  Since the tensor product of two indecomposable vector bundles is a direct sum of indecomposable vector bundles all of the same slope, and the value of that slope is the sum of the slopes of $\mathcal{L}_{i}$ and $\mathcal{L}_{j}^{*}$, each summand has slope $\mu_{i}-\mu_j$, which is less than zero by Lemma \ref{lemma.slopeinc}.  Since such indecomposable bundles have no global sections, the result follows.
\end{proof}


Let $V := \Hom(\cL_{-1}, \cL_0)$.  From Corollary~\ref{cor.main} there is an algebra homomorphism $\mathbb{S}^{nc}(V) \arr B_{\underline{\cL}}$.

\begin{prop} \label{prop.Bgendeg2}
The algebra $B_{\underline{\cL}}$ is generated in degrees one and two. In fact, $B_{\underline{\cL}}$ is generated as a left or right $\mathbb{S}^{nc}(V)$-module in degrees zero and two.
\end{prop}
\begin{proof}
By Lemma \ref{lemma.ext} and the proof of Proposition \ref{prop.extzero}, for all $i$ and all $j \geq 3$, multiplication $B_{i,i+1} \otimes B_{i+1, i+j} \rightarrow B_{i, i+j}$ is surjective.  This implies the first statement, as well as the fact that $B_{\underline{\cL}}$ is generated as a left $\mathbb{S}^{nc}(V)$-module in degrees zero and two.  It remains to check that $B_{\underline{\cL}}$ is generated as a right $\mathbb{S}^{nc}(V)$-module in degrees zero and two.  To this end, let $\underline{\mathcal{L}}^{*}$ denote the dual helix $\ldots, \cL_1^*, \cL_0^*, \cL_{-1}^*,\ldots$ and let $B_{\underline{\mathcal{L}}^{*}}$ denote the corresponding orbit algebra.  The same argument as in the first paragraph implies that, if we write $B^{*}:=B_{\underline{\cL}^*}$, then for all $i$ and all $j \geq 3$, multiplication $B_{i,i+1}^* \otimes B_{i+1, i+j}^* \rightarrow B_{i, i+j}^*$ is surjective.  On the other hand, for each $i,j \in \mathbb{Z}$, taking duals gives an isomorphism of vector spaces $(B_{\underline{\mathcal{L}}})_{i,j} \longrightarrow (B_{\underline{\mathcal{L}}^{*}})_{-j,-i}$ which reverses the order of composition.  It follows that, for all $i$ and all $j \geq 3$, multiplication $B_{-i-j, -i-1} \otimes B_{-i-1, -i} \rightarrow B_{-i-j,-i}$ is surjective, and the result follows.


\end{proof}

We next examine the Hilbert series of $B:= B_{\underline{\cL}}$.  Let $s_n:= \dim_{k} \mathbb{S}^{nc}(V)_{i,i+n}$ for $i \in \mathbb{Z}$.  Then the Euler exact sequence \cite[Corollary 3.5]{abstractp1} gives $h_{e_{i}\mathbb{S}^{nc}(V)}(t) = \sum_{n\geq 0} s_n t^n = \frac{1}{1 -dt + t^2}$, which is independent of $i$. Let us fix some index $i$ and let $h_{B}(t) := \sum_{n\geq 0} (\dim_k B_{i,i+n})t^n$ be the Hilbert series of $e_iB$.

\begin{prop}  \label{prop.HilbElliptic}
The Hilbert series $h_B(t) = \frac{1+t^2}{1-dt+t^2}$.
\end{prop}
\begin{proof}
Let $b_n := \chi(\cL_{-i-n}^* \otimes \cL_{-i})$ so for $n \geq 1$, Lemma~\ref{lemma.ext} ensures that $b_n = \dim_k B_{i,i+n}$.  Since $X$ is an elliptic curve,
$$ b_n = \deg (\cL_{-i-n}^* \otimes \cL_{-i}) = r_{-i-n}d_{-i} - r_{-i}d_{-i-n}.$$
The Euler exact sequence and additivity of rank and degree show that the sequences of $\{d_n\}$ and $\{r_n\}$ both satisfy the recurrence relation
\begin{equation} \label{eq.recur}
a_{n+1} - d a_n + a_{n-1} = 0.
\end{equation}
Furthermore, the sequence $\{s_n\}$ satisfies (\ref{eq.recur}) for $n\geq 0$ \cite[Theorem 3.4]{abstractp1}. Our formula for $b_n$ also shows that $\{b_n\}$ satisfies (\ref{eq.recur}). We calculate some initial values for $s_n$.
$$ s_0 = 1, \ s_1 = d, \ s_2 = d^2 -1, \ s_3 = d^3 -2d.$$
The initial values for $b_n$ are
$$ b_0 = 0, \ b_1 = d, \ b_2 = d^2 = s_2 + s_0, \ b_3 = d^3-d = s_3 + s_1.$$
Now solutions to (\ref{eq.recur}) are completely determined by two consecutive values so we see immediately that $b_n = s_n + s_{n-2}$ for $n \geq 2$. The proposition now follows.
\end{proof}

\begin{cor}  \label{cor.Bdoublecover}
For any index $i$, we have a direct sum decomposition of $\mathbb{S}^{nc}(V)$-modules $e_i B = e_i \mathbb{S}^{nc}(V) \oplus f_{i+2} \mathbb{S}^{nc}(V)$ where $f_{i+2} \in B_{i,i+2}$ is any element not in $e_i \mathbb{S}^{nc}(V)_{i,i+2}$. In particular, the morphism $\mathbb{S}^{nc}(V) \arr B$ is injective.
\end{cor}
\begin{proof}
Our computations in the proof of Proposition~\ref{prop.HilbElliptic} show that $\dim_k B_{i,i+2} = \dim_k \mathbb{S}^{nc}(V)_{i,i+2}+1$. Hence from Proposition~\ref{prop.Bgendeg2}, we know that multiplication induces a surjective $\mathbb{S}^{nc}(V)$-module morphism $e_i \mathbb{S}^{nc}(V) \oplus f_{i+2} \mathbb{S}^{nc}(V) \arr e_i B$. Proposition~\ref{prop.HilbElliptic} shows that it is also injective since the domain and codomain have the same Hilbert series.
\end{proof}
This corollary is known in the case $d=2$ and the helix is generated by line bundles, where it essentially reflects the fact that the complete linear system of a degree two line bundle expresses the elliptic curve as a double cover of $\mathbb{P}^1$ ramified at 4 points.

\begin{cor} \label{cor.nonnoetherian}
If $d>2$, then the $\mathbb{Z}$-algebra $B_{\underline{\mathcal{L}}}$ is nonnoetherian.
\end{cor}

\begin{proof}
The result follows from Corollary \ref{cor.Bdoublecover} and \cite[Theorem 1.2]{subex} in light of the fact that $\mathbb{S}^{nc}(V)$ is nonnoetherian.
\end{proof}

\begin{proposition}  \label{prop.limitmu}
Let $\mu_m$ denote the slope of $\cL_m$.  Then, as $m \arr -\infty$,
$$ \mu_m \arr -\frac{2d}{d-2 + \sqrt{d^2-4}}.$$
\end{proposition}
\begin{proof}
Consider the dual helix $\ldots, \cL_1^*, \cL_0^*, \cL_{-1}^*,\ldots$ and let $d_m = \deg \cL_{-m}^*, r_m = \text{rank}\, \cL_{-m}^*$. We can compute these using the recurrence relation (\ref{eq.recur}). Let $\alpha_{\pm} = \frac{1}{2}(d \pm \sqrt{d^2-4})$ be the roots of the characteristic equation so
$$ d_m = d_+ \alpha_+^m + d_-\alpha_-^m, \quad r_m = r_+ \alpha_+^m + r_-\alpha_-^m$$
where initial conditions give
$$ d_+ = \frac{d\alpha_+}{\alpha_+ - \alpha_-}, \quad r_+ = \frac{\alpha_+-1}{\alpha_+ - \alpha_-}.$$
But
$$
\lim_{m \arr \infty} \frac{d_{m}}{r_{m}} =
\lim_{m\arr \infty} \frac{d_+ \alpha_+^m + d_-\alpha_-^m}{r_+ \alpha_+^m + r_-\alpha_-^m} = \frac{d_+}{r_+}= \frac{d\alpha_+}{\alpha_+ -1} = \frac{2d}{d-2 + \sqrt{d^2-4}}.
$$
The result follows from the fact that if $\mathcal{L}$ is an indecomposable vector bundle over $X$ of degree $d$, then the degree of $\mathcal{L}^{*}$ is $-d$.
\end{proof}

We refer to the limiting slope in Proposition \ref{prop.limitmu} as $\theta_{d}$.

Now we recall some terminology from \cite{polish1}.  For $\theta \in \mathbb{R}$, Polishchuk defines a nonstandard $t$-structure on $D^{b}(X)$, and defines the abelian category ${\sf C}^{\theta}$ as the heart of this $t$-structure.  He then proves that ${\sf C}^{\theta}$ has cohomological dimension one and is derived equivalent to $X$ (\cite[Proposition 1.4]{polish1}).  Furthermore, he shows in \cite{polish3} that ${\sf C}^{\theta}$ is equivalent to the category of holomorphic vector bundles on the noncommutative two-torus $T_{\theta}$ (see \cite[Section 2.1]{polish1} for a definition of $T_{\theta}$).

\begin{prop} \label{prop.polish}
The $\mathbb{Z}$-algebra $B_{\underline{\mathcal{L}}}$ is coherent and ${\sf cohproj }B_{\underline{\mathcal{L}}}$ is equivalent to the noncommutative elliptic curve ${\sf C}^{\theta_{d}}$.
\end{prop}

\begin{proof}
By Proposition \ref{prop.limitmu}, $\underline{\mathcal{L}}$ is a sequence such that as $m \arr -\infty$, $\mu_{m} \arr \theta_{d}$.  In addition, by Lemma \ref{lemma.slopeinc}, $\theta_{d} < \mu_{m}$ for all $m$, and by Lemma \ref{lem:EllipticRightMutate}, $\operatorname{End}(\mathcal{L}_{m})=k$ for all $m$. Therefore, by the proof of \cite[Theorem 3.5]{polish1}, the sequence $\underline{\mathcal{L}}$ satisfies the conditions $(i)$ and $(ii)$ from the proof of \cite[Theorem 3.2]{polish1}.  It follows that $\underline{\mathcal{L}}$ is an ample sequence in the sense of \cite[Section 2.3]{abstractp1}, so that $B_{\underline{\mathcal{L}}}$ is coherent by \cite[Proposition 2.3]{polish}, and ${\sf cohproj }B_{\underline{\mathcal{L}}}$ is equivalent to ${\sf C}^{\theta_d}$ by \cite[Section 2.3]{abstractp1}.
\end{proof}

\begin{remark}
Since $\theta_{d}$ is a quadratic irrationality for $d>2$, \cite[Theorem 3.4]{polish1} implies that $C^{\theta_{d}}$ has a $\mathbb{Z}$-graded homogeneous coordinate ring.  Furthermore, this ring is sometimes better behaved than the ring $B_{\underline{\mathcal{L}}}$.  For example, if $d=3$ or $d=4$, one can show, using the proof of \cite[Theorem 3.4]{polish1}, that $C^{\theta_{d}}$ has a Koszul coordinate ring.  In contrast to this, Proposition \ref{prop.Bgendeg2} implies that $B_{\underline{\mathcal{L}}}$ is not even a quadratic algebra.
\end{remark}

\begin{proposition} \label{prop.coh}
The functor $f_{*}:{\sf Gr }B_{\underline{\mathcal{L}}} \rightarrow {\sf Gr } \mathbb{S}^{nc}(V)$ induced by Corollary \ref{cor.main} restricts to a functor on coherent modules
$$
{\sf coh }B_{\underline{\mathcal{L}}} \rightarrow {\sf coh }\mathbb{S}^{nc}(V).
$$
\end{proposition}

\begin{proof}
Let $M$ denote a coherent $B_{\underline{\mathcal{L}}}$-module.  Since $f_{*}$ is exact, the finite presentation for $M$ induces a map of the form
$$
\oplus_{i\in I}f_{*}e_{i}B_{\underline{\mathcal{L}}} \rightarrow \oplus_{j\in J}f_{*}e_{j}B_{\underline{\mathcal{L}}}
$$
with $I$ and $J$ finite subsets of integers, and cokernel isomorphic to $f_{*}M$.  The coherence of $f_{*}M$ now follows from Corollary \ref{cor.Bdoublecover} since $e_{i}\mathbb{S}^{nc}(V)$ is a coherent module for all $i \in \mathbb{Z}$.  The result follows from this.
\end{proof}

\begin{cor} \label{cor.pointind}
Let $p_{1},p_{2} \in X$ denote closed points.  Suppose $\underline{\mathcal{L}}_{i}$ is the helix on $X$ generated by the pair $(\mathcal{O}, \mathcal{O}(d\cdot p_{i}))$ with $d \geq 2$.  Then $B_{\underline{\mathcal{L}}_{1}} \cong B_{\underline{\mathcal{L}}_{2}}$.
\end{cor}

\begin{proof}
This follows immediately from Lemma \ref{lemma.functorringmap} using $F=t_{*}$ where $t:X \rightarrow X$ is translation by $p_{2}-p_{1}$.
\end{proof}



We conclude with a proof of Theorem \ref{thm.main}.  Part (1) follows from Theorem \ref{thm:ExistHelix} and Corollary \ref{cor.pointind}.  Part (2) follows from Corollary \ref{cor.nonnoetherian} and Proposition \ref{prop.polish}.  Part (3) follows from Proposition \ref{prop.ncsym}, \cite[Theorem 5.4]{bel} and \cite[Corollary 1.3]{abstractp1}. To prove Part (4), we first note that the fact that the hypotheses of Theorem \ref{thm.makemap} hold follows from Lemma \ref{lemma.ext} and its proof, along with Remark \ref{remark.extfinite}.  The fact that the functor (\ref{eqn.geomap}) restricts to a functor ${\sf C}^{\theta_{d}} \longrightarrow \mathbb{P}^{1}_{d}$ follows from Proposition \ref{prop.coh} and the proof of \cite[Lemma 2.2(1)]{abstractp1}, and the fact that it is a double cover follows from Corollary \ref{cor.Bdoublecover}.  The final assertion in (4) follows from Example \ref{example.commagree}.

\end{document}